\documentclass{article}

\usepackage{amssymb}
\usepackage{amsmath}
\usepackage{amsthm}
\usepackage{bbm}
\usepackage{graphicx}
\usepackage{tabularx}

\newcommand\groundset{\Omega}
\newcommand\allsets{\mathcal{P}\left(\groundset\right)}
\newcommand\dign{\mathcal{S}_n}
\newcommand\field{\mathbb{R}}
\newcommand\perms{\mathfrak{S}_n}
\newcommand\firstcom{\textrm{Com}\left(\perms\right)}
\newcommand\kdign{\dign^{\otimes k}}
\newcommand\com[1]{\textrm{Com}_{#1}\left(\perms\right)}
\newcommand\gorb[1]{\mathsf{B}_{#1}}
\newcommand\bl[1]{\mathsf{#1}}

\newcommand\permsv{\mathfrak{S}_V}
\newcommand\permse{\mathfrak{S}_E}
\newcommand\nvert{f}
\newcommand\galg{\mathcal{G}_\nvert}
\newcommand\ideal{\mathcal{I}_\nvert}

\begin{document}

\newtheorem{definition}{Definition}
\newtheorem{theorem}{Theorem}
\newtheorem{lemma}{Lemma}
\newtheorem{conjecture}{Conjecture}
\newtheorem{remark}{Remark}
\newtheorem{corollary}{Corollary}

\title{On some combinatorial properties of the orbits on subsets}
\author{Xavier Buchwalder\\
CWI Amsterdam \\ \small{xpjb@cwi.nl}}

\maketitle

\begin{abstract} We introduce generalised orbit algebras. The purpose here is to measure how some combinatorial properties can characterize the action of a group of permutations of the elements of $\groundset$ on the subsets of $\groundset$. The similarity with orbit algebras is such that it took the author a long time to find a generalised orbit algebra not arising from a permutation group.

\end{abstract}

\section{Introduction}

Let $\groundset$ be a finite set and $\allsets$ its powerset. 
For every set $A$, we shall denote by $\vert A\vert$ the number of elements in $A$.
We define a \emph{strongly regular partition} of $\allsets$ to be any partition of $\allsets$ into blocs $\gorb{1},\ldots,\gorb{s}$ such that~:

\begin{enumerate}
\item The blocs contain sets of the same size.
\item For any bloc $\gorb{i}$, the set of all the complements of the members of $\gorb{i}$ is also a bloc, say $\gorb{c(i)}$.
\item For any two blocs $\gorb{i}$ and $\gorb{j}$, the number of members of $\gorb{j}$ that are included in a member $A$ of $\gorb{i}$ is independent of the choice of $A$ in $\gorb{i}$. We denote by $\binom{\gorb{i}}{\gorb{j}}$ this number.
\end{enumerate}

We first remark that this definition is somewhat equivalent to having a tactical decomposition between each pair of blocs. As an example, if $\groundset=\{1,2,3\}$, the following partition is strongly regular :

\begin{center}
\begin{tabular}[width=\textwidth]{cccccc}
$\emptyset$ & $\{1\},\{2\}$ &$\{3\}$& 
$\{1,2\}$& $\{1,3\},\{2,3\}$&
 $\{1,2,3\}$\\
\end{tabular}
\end{center}

The main concern of the paper is the study of strongly regular partitions, through an algebraic point of view. A long standing question for the author was whether there exists a strongly regular partition not arising as the set of orbits of a group of permutations of $\groundset$ on its powerset. This question is answered at the end of the paper.

Along the way, various results about the reconstruction conjecture in graph theory are shown to apply for strongly regular partitions. Among them the results of L.Lov\'asz and V.M\"uller play a key role : our point of view is to try to explain and extend these results and the counting strategy behind them. For this purpose, we prove the following :

\begin{lemma}
For any two blocs $\gorb{i}$ and $\gorb{j}$, the number of members of $\gorb{j}$ that contain a member $A$ of $\gorb{i}$ is independent of the choice of $A$ in $\gorb{i}$. Moreover, this number is equal to $\frac{\binom{\Omega}{\gorb{j}}}{\binom{\Omega}{\gorb{i}}}\binom{\gorb{j}}{\gorb{i}}$ and to $\binom{\gorb{c(i)}}{\gorb{c(j)}}$.
\end{lemma}

\begin{theorem}
If $P$ is a strongly regular partition, then for any two blocs $\gorb{i}$ and $\gorb{j}$ of $P$,  
the number of members of $\gorb{j}$ that intersects a subset of size $r$ of a given member $A$ of $\gorb{i}$ is independent of the choice of $A$ in $\gorb{i}$.
\end{theorem}

\begin{theorem}
If $P$ is a strongly regular partition, then for any three blocs $\gorb{i}$, $\gorb{j}$, and $\gorb{k}$ of $P$, and any integers $r_1,\ldots,r_4$~: 
\[ \left\vert \left\{ (A,B)\in \gorb{i}\times \gorb{j}\ s.t.\ 
\left\{
\begin{array}{l} \vert A\cap B\vert=r_1\\ 
\vert A\cap C\vert=r_2\\
\vert B\cap C\vert=r_3\\
\vert A\cap B\cap C\vert=r_4
\end{array}\right.
\right\}\right\vert\]
is independent of the choice of $C$ in $\gorb{k}$.
\end{theorem}

One could have expected that these new results, and the framework that comes along would yield an improvement on the Edge Reconstruction Conjecture. A discussion of this problem is provided, including some previously unknown limit cases that establish the impossibility of improving the result of V.M\"uller in the general case of reconstruction under group action.

We choose the algebraic point of view for our exposition, and instead of the pair orbits-blocs, we study orbit algebras and the generalised orbit algebras corresponding to strongly regular partitions. This emphasize the role of two operators : derivation and complementation, which are shown to generate the Terwilliger algebra. This shows a link between the Terwilliger algebra and the reconstruction conjectures that was never emphasized to this degree.

\section{Generalised orbit algebras}

\subsection{Orbit algebras}

\begin{definition} We consider $\dign$ the quotient algebra of the polynomial algebra $\field[x_1,\ldots,x_n]$ by the ideal of polynomials  generated by $x_1^2-x_1,x_2^2-x_2,\ldots,x_n^2-x_n$.
\end{definition}

To emphasize the product of $\dign$, we sometimes write $p\cdot q$ instead of $pq$. 
It is clear that $\dign$ is a real vector space of dimension $2^n$, and that the polynomials 

\[p_A=\prod_{i\in A} x_i,\;\mathrm{ with }\; A\in\mathcal{P}\left(\groundset\right)\]

are a basis of $\dign$. We observe that~:

\begin{itemize}
\item $p_{\emptyset}=1$
\item $\forall \left(A,B\right) \in \allsets^2,\quad p_A\cdot p_B=p_{A \cup B}$
\end{itemize}

For every $n$, the polynomial algebra $\dign$ is also an algebra of functions on $\allsets$ (with pointwise multiplication)~:

\begin{definition} If $A$ and $B$ are two subsets of $\groundset$, we define the value of $p_A$ at the set $B$ to be~:
\[p_A\left(B\right)=\left\{\begin{array}{l}1\ \textrm{if}\ A\subseteq B\\0 \ \textrm{otherwise} \end{array}\right.\]
\end{definition}

As the $p_A$, $A\in\allsets$, form a basis of $\dign$, we can extend the above evaluation to every polynomial of $\dign$ by linearity. We call the function which maps a polynomial of $\dign$ to its associated real function on $\allsets$ \emph{evaluation}. It is an algebra isomorphism, as~:

\begin{itemize}
\item For every triple $\left(A,B,C\right)$ of subsets of $\groundset$, $A\cup B\subseteq C$ if and only if $A\subseteq C$ and $B\subseteq C$, hence $p_A\cdot p_B\left( C\right)=p_{A\cup B}\left(C\right)=p_A\left(C\right)p_B\left(C\right)$.
\item If $p=\sum_{A\subseteq\groundset}\alpha_A p_A$ is such that for every subset $B$ of $\groundset$ $p\left(B\right)=0$, then $p$ is identically $0$. To prove this, we consider a subset $C$ of $\groundset$ such that $\alpha_C\neq 0$ and $\vert C\vert$ is maximal for this property. Then we would have $p\left(C\right)=0$ and $p\left(C\right)=\alpha_C$.
\end{itemize}

This isomorphism indicates that Lagrange interpolation can be used, and we shall use it to deduce our first structure theorem concerning the subalgebras of $\dign$~:

\begin{theorem}\label{partition} For every subalgebra $\mathcal{A}$ of $\dign$ that contains $1$, there exists a partition $\bl{P}$ of $\allsets$ such that~:
\[\mathcal{A}=\left\{p\ \in \mathcal{S}_n\quad s.t.\quad \forall \bl{P}_i \in \bl{P}\quad \forall (A,B)\in \bl{P}_i^2 \quad p(A)=p(B)   \right\}\]
\end{theorem}

\begin{proof}
For every polynomial $p$, we define an equivalence relation on $\allsets$ by~: \[A \mathcal{R}_p B \Leftrightarrow p(A)=p(B)\]
This relation partitions $\allsets$ into generalised orbits $\bl{P}_1,\ldots,\bl{P}_s$, corresponding to distinct values of $p$. We put 
\[r_i=\frac{\prod_{T \notin \bl{P}_i}\left(p-p(T)\right)}{\prod_{T \notin \bl{P}_i}\left(p(A)-p(T)\right)} \quad \textrm{for some}\ A\in \bl{P}_i\]
and observe that $r_i\left(B\right)$ is $1$ if $B$ is a member of $\bl{P}_i$ and $0$ otherwise. Also, each $r_i$ is in $\mathcal{A}$. If we consider the total relation, that is  
\[A \mathcal{R} B \Leftrightarrow p(A)=p(B),\; \forall p \in \mathcal{A}\]
then $\mathcal{R}$ partitions $\allsets$ into generalised orbits $\bl{Q}_1,\ldots,\bl{Q}_s$, and we see that 
\[\mathcal{A}\subseteq\left\{p\ \in \mathcal{S}_n\quad s.t.\quad \forall i=1\ldots s,\ \forall (A,B)\in \bl{Q}_i^2 \quad p(A)=p(B)   \right\}\]
As for the reverse inclusion, there exists a polynomial $\varepsilon_i$ in $\mathcal{A}$ such that $\varepsilon_i\left(A\right)$ is $1$ if $A$ is a member of $\bl{Q}_i$ and $0$ otherwise~: we consider the product of some of the $r_i$ defined above.
The isomorphism between $\dign$ and real functions on $\allsets$ allows us to conclude that any polynomial that takes constant values on each of the $\bl{Q}_i$ is a linear combination of the $\varepsilon_i$, that is to say, is in $\mathcal{A}$.
\end{proof}

We observe that, conversely, any partition $\bl{P}$ of $\allsets$ uniquely defines a subalgebra of $\dign$ containing $1$, namely~:
\[\left\{p\ \in \mathcal{S}_n\quad s.t.\quad \forall \bl{P}_i \in \bl{P}\quad \forall (A,B)\in \bl{P}_i^2 \quad p(A)=p(B)   \right\}\]

and deduce that there is a finite number of subalgebras of $\dign$~: exactly as many as there are partitions of $\allsets$.

The group of permutations $\perms$ of the elements of $\groundset$ acts on $\dign$ as a group of algebra isomorphisms defined by~:
\[\sigma\cdot x_i= x_{\sigma(i)}\;,\;\sigma\ \in \ \perms\]

\begin{definition}
For every subgroup $\Gamma$ of $\perms$, the set
\[\dign^\Gamma=\left\{p \in \dign\quad s.t.\quad \sigma\cdot p = p\quad\forall \sigma\in \Gamma \right\}\]
is a subalgebra of $\dign$ called the \emph{orbit algebra} of $\Gamma$. Its elements are called the \emph{invariants} of $\Gamma$. If an algebra $\mathcal{A}$ is equal to $\dign^\Gamma$ for some subgroup $\Gamma$ of $\perms$, we will say that $\mathcal{A}$ is an \emph{orbit algebra}.
\end{definition}

The above considerations yield a natural enumeration of the subalgebras of $\dign$ as a bijection with the partitions of $\allsets$, and we would like to obtain a combinatorial characterization of subalgebras of $\dign$ that are orbit algebras, or equivalently, of partitions of $\allsets$ that are the orbits of a subgroup $\Gamma$ of $\perms$.

To conclude with orbit algebras, we remark that two distinct groups do not necessarily yield the same orbit algebra but it is nevertheless almost always true for primitive permutation groups, according to a theorem of \cite{cameron2}. 

\subsection{Generalised orbit algebras}

\begin{definition}
\begin{itemize}
\item We call \emph{derivation} the linear mapping $\partial$ of $\dign$ to itself defined by~:
\[\partial\left( p_A\right) = \sum_{i\in A} p_{A\setminus i},\quad \forall A\subseteq \groundset\]
\item We call \emph{complementation} the linear mapping $\complement$ of $\dign$ to itself defined by~:
\[\complement\left( p_A\right) = p_{\groundset\setminus A},\quad \forall A\subseteq \groundset\]
\end{itemize}
\end{definition}

We have for example that $\partial(p_{\emptyset})=\partial(1)=0$, and that $\complement(p_{\emptyset})=\complement(1)=p_{\groundset}$. Also, if $n\geq 2$, $\partial(p_{\{1,2\}})=p_{\{1\}}+p_{\{2\}}$.

It is clear that complementation is an involution, that is to say $\complement\circ\complement$ is the identity of $\dign$. As for $\partial$, one can show that for any integer $k$ and any subset $A$ of $\groundset$, 
\begin{eqnarray}\label{derivpow}
\partial^k(p_A)&=&\sum_{\substack{B \subseteq A\\ \vert B \vert =\vert A\vert -k}} k!p_{B} 
\end{eqnarray}
Hence, $\partial$ is nilpotent of order $n+1$.

One can show that for every permutation $\sigma$ in $\perms$, and every pair $p,q$ of polynomials in $\dign$, one has~:
\begin{itemize}
\item $\sigma\left(p\cdot q\right)=\sigma\left(p\right)\cdot\sigma\left(q\right)$
\item $\sigma\circ\partial\left(p\right)=\partial\circ\sigma\left(p\right)$
\item $\sigma\circ\complement\left(p\right)=\complement\circ\sigma\left(p\right)$
\end{itemize}

We deduce that any orbit algebra is closed under derivation, complementation and multiplication. This is underlying numerous results about orbit algebras and orbits on subsets as the Livingstone-Wagner theorem \cite{livingstone} : the number of orbits on $k+1$-subsets is at least the number of orbits on $k$-subsets if $k<\frac{n}{2}$. Also, the orbits on $k$-subsets are determined independently of the group by the orbits on $k+1$-subsets in this case, that is to say the restriction of $\partial$ to the orbits of size $k+1$ is surjective if $k<\frac{n}{2}$, and also injective if $k\geq\frac{n}{2}$. A simple proof is provided below. 

For instance, if $n=3$, one can consider the real vector space on  \[\left\{1,\ x_1+x_2,\ x_3,\ x_1x_2,\ x_1x_3+x_2x_3,\ x_1x_2x_3\right\}\] which is the orbit algebra of 
the subgroup $\left\{\textrm{Id},(1 2)(3)\right\}$. It is closed under multiplication, and we have~:
\[\begin{array}{rclrcl}

\partial \left(1\right)&=&0& \complement  \left(1\right)&=&x_1x_2x_3\\
\partial \left(x_1+x_2\right)&=&2&\complement \left(x_1+x_2\right)&=&x_1x_3+x_2x_3\\
\partial \left(x_3\right)&=&1&\complement  \left(x_3\right)&=&x_1x_2\\
\partial \left(x_1x_2\right)&=&x_1+x_2&\complement \left(x_1x_2\right)&=&x_3\\
\partial \left(x_1x_3+x_2x_3\right)&=&x_1+x_2+2x_3&\complement \left(x_1x_3+x_2x_3\right)&=&x_1+x_2\\
\partial \left(x_1x_2x_3\right)&=&x_1x_2+x_1x_3+x_2x_3&\complement \left(x_1x_2x_3\right)&=&1\\

\end{array}\]

It has been a long standing question for the author to find out whether there are subspaces of $\dign$ closed under derivation, complementation and multiplication that are not the orbit algebras of a group of permutations. Computer aided enumeration indicates that they do not exist with $n\leq 6$. This question motivated the following developments.

\begin{definition}
A \emph{generalised orbit algebra} is a nonempty nonzero subalgebra of $\dign$ that is closed under derivation and complementation. To emphasize the parameter $n$, we might say that a generalised orbit algebra has \emph{order} $n$.
\end{definition}

As generalised orbit algebras are particular subalgebras of $\dign$ (we will see that they always contain $1$), one might consider their associated partition defined by Theorem \ref{partition}. That is to say, given a generalised orbit algebra $\mathcal{D}$, the subsets of the set $\allsets$ on which every polynomial of $\mathcal{D}$ take the same value. For the above example, one can see that the generalised orbits of the associated partition are~: 
\begin{center}
\begin{tabular}[width=\textwidth]{cccccc}
$\emptyset$ & $\{1\},\{2\}$ &$\{3\}$& 
$\{1,2\}$& $\{1,3\},\{2,3\}$&
 $\{1,2,3\}$\\
\end{tabular}
\end{center}
It is a remarkable fact (for a subalgebra of $\dign$), that there is a natural bijection between the generalised orbits of this partition and the polynomial basis we took for this particular algebra. This is indeed true of any orbit algebra\footnote{i.e. the orbits of a group $\Gamma$ on $\allsets$ index a basis of the orbit algebra of $\Gamma$, and form the partition associated with $\dign^{\Gamma}$ by Theorem \ref{partition}.}, and more generally, of any generalised orbit algebra. In the following, we develop a formal algebraic machinery which can be applied to establish this fact (although generalised orbit algebras are not the only algebras satisfying this property).

\begin{definition} We denote by $l$ the linear mapping of $\dign$ to itself defined by~:
\[l(p)=\sum_{k=0}^{n}\frac{\partial^k (p) }{k!} \]
\end{definition}
Observing the equation (\ref{derivpow}), one can see that $\forall A\subseteq \groundset$,
\[l(p_A)=\sum_{B\subseteq A} p_B\]

The powers of $l$ have a neat expression in terms of $\partial$~:

\begin{lemma}[Mnukhin \cite{mnukhin2}]
 \label{lpuiss} For every nonzero integer $m$~:
\[l^m=\sum_{k=0}^{n} \frac{m^k}{k!}\partial^k \]
\end{lemma}

\begin{proof}
If $r$ and $s$ are nonzero integers, one has~:
\begin{eqnarray*}
\left(\sum_{k=0}^{n} \frac{r^k}{k!}\partial^k\right) \circ \left(\sum_{i=0}^{n} \frac{s^i}{i!}\partial^i \right)& = &
\sum_{k=0}^{n}\sum_{i=0}^{n}  \frac{r^k s^i}{k! i!}\partial^{k+i} \\
&=& \sum_{m=0}^{2n}\partial^{m} \sum_{k=0}^{m} \frac{r^k s^{m-k}}{k! (m-k)!}\\
&=& \sum_{m=0}^{n}\frac{\partial^{m}}{m!}\sum_{k=0}^{m}\frac{m!}{k! (m-k)!} r^k s^{m-k}\\
&=& \sum_{m=0}^{n}\frac{(r+s)^m}{m!}\partial^{m}
\end{eqnarray*}

As the Lemma is true if $m=1$, the above equality provides a proof of the property for $m>0$, by induction. One can use the same computations with $r=-1$ and $s=1$ to show that~: 
\[\left(\sum_{k=0}^{n} \frac{(-1)^k}{k!}\partial^k\right) \circ l = Id\]
The Lemma is then true if $m=-1$, and thus if $m<0$ by induction.
\end{proof}

This Lemma implies that $l$ is an isomorphism, and by using a VanderMonde matrix, we see that there exist rational numbers $a_1,\ldots,a_{n+1}$ such that~:
\begin{eqnarray}\label{dwithl}
\partial=\sum_{r=1}^{n+1} a_r l^r
\end{eqnarray}

\begin{definition} We denote by $\varepsilon$ the linear mapping of $\dign$ to itself defined by~: 
\[\varepsilon(p)=\complement \circ l^{-1} \circ \complement (p)\]
Moreover, we shall use the simplified notation $\varepsilon_A=\varepsilon\left(p_A\right)$ for any subset $A$ of $\groundset$.
\end{definition}

The mapping $\varepsilon$ is an isomorphism, hence $\left\{\varepsilon_A\ :\ A\subseteq\groundset\right\}$ is a basis of $\dign$. By equation (\ref{derivpow}), and Lemma \ref{lpuiss} we have~:
\begin{equation}
\label{formuleeps}\varepsilon_A=\sum_{B\supseteq A} (-1)^{\vert B\vert -\vert A\vert}p_B
\end{equation}

\begin{lemma}
\label{evalepsilon} $\varepsilon_A(B)=1$ if $A=B$ and $0$ otherwise .
\end{lemma}

\begin{proof}
\begin{eqnarray*}
\varepsilon_A(B)&=&\sum_{C\supseteq A} (-1)^{\vert C\vert - \vert A\vert}p_C(B)\\
&=&(-1)^{\vert A\vert} \sum_{C\ s.t.\ A\subseteq C\subseteq B} (-1)^{\vert C\vert}\\
&=&\left\{\begin{array}{c} 1\ if\ A=B \\0\ if\ A \neq B\end{array}\right.\\
\end{eqnarray*}
\end{proof}

Due to the existence of the isomorphism between $\dign$ and the real functions on $\allsets$, one can see that the $\varepsilon_A$ form a basis of idempotents of $\dign$, that is $\varepsilon_A\cdot\varepsilon_B=\varepsilon_A$ if $A=B$ and $0$ otherwise. 
The idea of constructing the idempotents in this way is well known, for example in \cite{mnukhin}. 
We immediately deduce the following~:

\begin{corollary}\label{isomorph}
Let $\mathcal{F}$ be the algebra of real-valued functions on $\allsets$ with pointwise multiplication. The following map is an algebra isomorphism between $\mathcal{F}$ and $\dign$~:
\[\theta\ :\ f\mapsto \sum_{A\subseteq\groundset} f(A)\varepsilon_A\]
Moreover, $\theta$ is the inverse of the evaluation map.
\end{corollary}
\begin{proof}According to Lemma \ref{evalepsilon}, the evaluation of $\theta\left(f\right)$ is $f$.\end{proof}

This is very useful for computing products in $\dign$.  In the following, we develop the link  between the polynomials $\varepsilon_A$ and the partition structure, but first we prove that every generalised orbit algebra contains 1, and more precisely, that every generalised orbit algebra contains a particular subspace~: the orbit algebra of $\perms$.

\begin{lemma}\label{card} Every generalised orbit algebra $\mathcal{D}$ contains the following polynomials~:
\[\sum_{A\ s.t.\ \vert A \vert=k} p_A,\quad k\in\{0\ldots n\}\]
\end{lemma}

\begin{proof} $\mathcal{D}$ being a generalised orbit algebra, there exists
\[p=\sum_{A\subseteq \groundset}\alpha_A\varepsilon_A \neq 0\ \in \mathcal{D}\]
As we have seen, the $\varepsilon_A$ are idempotents, thus~:
\[p^2=\sum_{A\subseteq \groundset}\alpha_A^2\varepsilon_A \in \mathcal{D}\]
and $\varepsilon$ is a linear isomorphism, so~:
\[\varepsilon^{-1}(p^2)=\sum_{A\subseteq \groundset}\alpha_A^2 p_A \in \mathcal{D}\]

Now if we let $k=max\left\{r\ /\ \exists A\subseteq\groundset\ s.t.\ \vert A \vert =r\ and\ \alpha_A\neq 0\right\}$, then 

\[\partial^k\circ\varepsilon^{-1}(p^2)=\underbrace{\left(\sum_{A\subseteq \groundset\ s.t. \ \vert A\vert=k}k!\alpha_A^2\right)}_{\neq 0} p_{\emptyset} \in \mathcal{D}\]

We have $1=p_{\emptyset} \in \mathcal{D}$, and $\mathcal{D}$ is closed under complementation, so $p_{\groundset} \in \mathcal{D}$. We conclude by noting that~:
\[\frac{1}{s!}\partial^s(p_{\groundset})=\sum_{A\ s.t.\ \vert A \vert=n-s} p_A,\quad s\in\{0\ldots n\}\]

\end{proof}
We now emphasize the link between the partition associated to a generalised orbit algebra by Theorem \ref{partition} with a natural basis of this generalised orbit algebra.

\begin{definition}
A \emph{generalised orbit} of a generalised orbit algebra $\mathcal{D}$ is a nonempty subset $\gorb{}$ of $\allsets$ such that
\begin{enumerate}
\item \[\sum_{A\in \gorb{}} p_A \in \mathcal{D}\]
\item $\gorb{}$ is minimal with respect to this property~: 
\[\forall\ \gorb{}'\subsetneq \gorb{},\ \gorb{}'\neq \emptyset\quad  \sum_{A\in \gorb{}'} p_A\; \notin\; \mathcal{D}\]
\end{enumerate}
\end{definition}

It is clear that two different generalised orbits $\gorb{1}$ and $\gorb{2}$ of $\mathcal{D}$ are disjoint, because if not, any generalised orbit algebra being closed under $\varepsilon$, $\gorb{1}\cap \gorb{2}$ would be a generalised orbit of $\mathcal{D}$, as~: 
 \[\left(\sum_{A\in \gorb{1}} \varepsilon_A\right)\left(\sum_{B\in \gorb{2}} \varepsilon_B\right)=
\sum_{C\in \gorb{1}\cap \gorb{2}} \varepsilon_C\]
This cannot happen, because $\gorb{1}$ and $\gorb{2}$ are minimal. The same argument shows that the set of all subsets $\gorb{}$ of $\allsets$ such that $\sum_{A\in \gorb{}} p_A \in \mathcal{D}$ is closed under intersection. Moreover, we have already seen\footnote{in Lemma \ref{card}} that the polynomials $\sum_{A\ s.t.\ \vert A \vert=k} p_A,\ k\in\{0\ldots n\}$, are in $\mathcal{D}$, thus any subset of $\groundset$ is in a generalised orbit.
We conclude that the generalised orbits of a generalised orbit algebra form a partition of $\allsets$.

\begin{theorem}\label{generalised orbits} Let $\mathcal{D}$ be a generalised orbit algebra, and $\gorb{1},\ldots,\gorb{s}$ the partition of $\allsets$ associated to $\mathcal{D}$ by Theorem \ref{partition}, that is~:
\[\mathcal{D}=\left\{p\ \in \mathcal{S}_n\quad s.t.\quad \forall i=1\ldots s\quad \forall (A,B)\in \gorb{i}^2 \quad p(A)=p(B)   \right\}\]
Then :
\begin{enumerate}
\item The associated set of polynomials $\varepsilon_{\gorb{i}}=\sum_{A\in \gorb{i}} \varepsilon_A$ is a basis of $\mathcal{D}$ as a real vector space.
\item The associated set of polynomials $p_{\gorb{i}}=\sum_{A\in \gorb{i}} p_A$ is a basis of $\mathcal{D}$ as a real vector space.
\item The generalised orbits of $\mathcal{D}$ are the sets $\gorb{1},\ldots,\gorb{s}$. 
\end{enumerate}
\end{theorem}

\begin{proof}
\begin{enumerate}
\item The $\varepsilon_A,\ A\subseteq\groundset$, constitute a basis of $\dign$ which contains $\mathcal{D}$, and the $\gorb{i}$ are disjoint, so the family is independent. If we consider an element $p$ of $\mathcal{D}$, we have, by Corollary \ref{isomorph}~: \[p=\sum_{A\subseteq\groundset}p(A)\varepsilon_A\]
Thus the family also generates $\mathcal{D}$.
\item As $\varepsilon$ is invertible, $\varepsilon^{-1}$ maps any basis of $\mathcal{D}$ to another basis of $\mathcal{D}$.
\item If $\sum_{A \in \gorb{}} p_A \in \mathcal{D}$ with $\gorb{}\subseteq \gorb{i}$ for some $i$, then if we put $q=\sum_{A\in\gorb{}}\varepsilon_A$, $q$ belongs to $\mathcal{D}$. Suppose that there exists both $C\in\gorb{i}\setminus \gorb{}$ and $D\in \gorb{}$. We would have $q(C)=0$ and $q(D)=1$, which would be a contradiction to the definition of $\gorb{i}$.
\end{enumerate}
\end{proof}

Various properties of this \emph{generalised orbit system} follow from Theorem \ref{partition}. For instance, if $\mathcal{D}_1$ and $\mathcal{D}_2$ are two generalised orbit algebras such that $\mathcal{D}_1\subseteq\mathcal{D}_2$, then the generalised orbits of $\mathcal{D}_2$ are a refinement of those of $\mathcal{D}_1$. According to Lemma \ref{card}, the generalised orbits of any generalised orbit algebra are a refinement of those of $\dign^{\perms}$, which means that all the sets in a generalised orbit $\gorb{}$ have the same size $\sharp\gorb{}$.

Generalised orbit algebras share the property of being closed under $\varepsilon$ with orbit algebras, that is, their generalised orbit system index one of their basis. From equation (\ref{dwithl}), we infer that generalised orbit algebras are exactly the nonempty and nonzero subalgebras of $\dign$ that are closed under $\varepsilon$ and $\complement$. For an example of an algebra closed under $\varepsilon$ but not $\complement$, one may consider the vector space spanned by~:
\[\left\{p_{\emptyset},\ p_{\{1\}}+p_{\{2\}},\ p_{\{3\}},\ p_{\{1,2\}},\ p_{\{1,3\}},\ p_{\{2,3\}},\ p_{\{1,2,3\}} \right\}\]
This basis is indexed by the generalised orbits of the partition given by Theorem \ref{partition} in a very natural way.

An important property of orbit algebras, and of the orbits of a permutation group on $\allsets$, is that the inclusion relations, and the stronger intersection relations have nice properties : for example, the number of elements of an orbit $\bl{O}_1$ that intersects a subset of size $r$ of a set $A$ depends only on the orbit $\bl{O}_2$ of $A$, and not on $A$ itself. In the following, we embark on a systematic study of such properties, our hope being to generalize them to generalised orbit algebras.

\begin{definition} We denote by $\firstcom$ the set of linear functions $h$ from $\mathcal{S}_n$ to itself such that~:
\[\forall \sigma \in \perms,\ h\circ\sigma = \sigma\circ h\]
\end{definition}

Given two pairs of subsets $(A_1,A_2)$ and $(B_1,B_2)$ of $\groundset$, one can see that there exists a permutation in $\perms$ that simultaneously maps $A_1$ to $B_1$ and $A_2$ to $B_2$ if and only if~:
\begin{itemize}
\item $\vert A_1\vert=\vert B_1\vert$
\item $\vert A_2\vert=\vert B_2\vert$
\item $\vert A_1 \cap A_2\vert=\vert B_1  \cap B_2\vert$
\end{itemize}

\noindent We deduce that the mappings 
\[E_{k,l,r}:p_A\mapsto \left\{\begin{array}{ll}\sum_{\substack{B\ s.t\ \vert B\vert =l\\ \vert A\cap B\vert=r}}p_B & \textrm{if}\ \vert A\vert=k\\0&\textrm{otherwise} \end{array} \right. \]
where $k,l,r$ are non-negative integers such that $r\leq k$, $r\leq l$, $k+l-r\leq n$, form a basis of $\firstcom$. As an example, if $k$ is an integer between $0$ and $n$, we shall denote by $id_k$ the function $E_{k,k,k}$ which maps a polynomial $p_A$ to itself if $\vert A\vert=k$, and to $0$ otherwise.

We observe that every orbit algebra is closed under any mapping in $\firstcom$. We will show that this is indeed true for any generalised orbit algebra with the following~:

\begin{theorem}$\firstcom$ is generated by $\partial$ and $\complement$ as a real algebra under composition of mappings.
\end{theorem}

We point out to the reader that $\firstcom$ is known as the Terwilliger algebra of the hypercube, and have already been the object of intensive study (see \cite {go}), with surprising applications (for example \cite {schrijver}). We independently provide here a set of generators that is more convenient for our purposes.

\begin{proof}
We will prove by induction that for every integer $k$ between $0$ and $\frac{n}{2}$, the $E_{u,v,w}$ where either $u\leq k$ or $v\leq k$ can be generated by $\partial$ and $\complement$, by composition, and by taking real linear combinations. 
\begin{itemize}
\item If $k=0$, we see that for every integer $l$~:
\[\partial^{n-l}\circ\complement\circ\partial^n\circ\complement= n! \left( n-l\right)! E_{0,l,0}\]
\[\partial^{l}\circ E_{0,0,0}=l!E_{l,0,0}\]
\item Suppose that for every integer $l$ less than $k\leq\frac{n}{2}$, we can generate all the $E_{u,v,w}$, with either $u\leq l$ or $v\leq l$.
Note that~:
\[\partial^{n-2k}\circ\complement=(n-2k)!\sum_{r=0}^{2k}E_{r,2k-r,0}\]
This allows us to construct $E_{k,k,0}$ as the only member of the right hand side not already addressed by the induction hypothesis.
By means of a right composition, we can generate the~:
\begin{equation}\label{derivcomp}
E_{k-1,k,t}\circ\partial=E_{k,k,t}+E_{k,k,t+1},\quad t=0\ldots k-1
\end{equation}
By induction, we construct $E_{k,k,t+1}$ for $t=0\ldots,k-1$.
In particular, $id_k=E_{k,k,k}$ is generated.
To conclude, we observe that~:
\begin{eqnarray}\label{bascom1}
\complement\circ\partial^{v-r}\circ\complement\circ\partial^{u-r}\circ id_u=(u-r)!(v-r)!\sum_{w=r}^{\min(u,v)} \binom{w}{r} E_{u,v,w} 
\end{eqnarray}
and that this is equal to 
\begin{eqnarray}\label{bascom2}
id_v\circ\complement\circ\partial^{v-r}\circ\complement\circ\partial^{u-r}=(u-r)!(v-r)!\sum_{w=r}^{\min(u,v)} \binom{w}{r} E_{u,v,w} 
\end{eqnarray}
For any $u$ and $v$ with either $u\leq k$ or $v\leq k$, we can now retrieve the $E_{u,v,w},\ w=0\ldots \min(u,v)$ by a triangular linear system.
\end{itemize}
\end{proof}

As previously mentioned, any generalised orbit algebra is therefore closed under any mapping of $\firstcom$. In other words, if $\mathcal{D}$ is a generalised orbit algebra with generalised orbits $\gorb{1},\ldots,\gorb{s}$, then for any generalised orbit $\gorb{i}$ the polynomial $E_{k,l,r}\left(p_{\gorb{i}}\right)$ is a member of $\mathcal{D}$. 
We see that~:
\[E_{k,l,r}\left(p_{\gorb{i}}\right)=\sum_{A\ s.t.\ \vert A\vert=l} \left\vert \left\{B \in \gorb{i}\ s.t.\ \vert B\cap A\vert=r \right\} \right\vert p_{A}\]
As this polynomial is a member of $\mathcal{D}$, we conclude that if $\gorb{i}$ and $\gorb{j}$ are two generalised orbits of $\mathcal{D}$, then for every set $A$ in $\gorb{j}$, the number of sets in $\gorb{i}$ that intersects a subset of size $r$ of $A$ is independent of the choice of $A$ in $\gorb{j}$.

We can also see that in any generalised orbit algebra, the restriction of $\partial$ to the orbits of size $k$ is injective if $k>\frac{n}{2}$. 
Indeed by equation \ref{derivcomp}, we have~: 
\[E_{k-1,k,t}\circ\partial=E_{k,k,t}+E_{k,k,t+1},\quad t=0\ldots k-1\]
If $k>\frac{n}{2}$, we have $E_{k,k,0}=0$ and thus~:
\[\sum_{t=0}^{k-1}(-1)^{k-1-t} E_{k-1,k,t}\circ\partial=(-1)^{k-1} E_{k,k,0}+E_{k,k,k}=id_k\]
so the restriction of $\partial$ to the orbits of size $k$ is injective. 

As $\partial\circ id_k=E_{k,k-1,k-1}$, we see that $\complement\circ\partial\circ\complement=E_{n-k,n-k+1,n-k+1}$ is injective. With $r=n-k$ this yields that $E_{r,r+1,r}$ is injective if $r<\frac{n}{2}$, and we remark that its transpose in the canonical basis is $E_{r+1,r,r}$ which is then surjective. In any generalised orbit algebra, the number of generalised orbits with cardinality $k+1$ is then at least the number of generalised orbits with cardinality $k$ if $k<\frac{n}{2}$, so to say the Livingstone-Wagner theorem applies.

We shall now investigate further this idea that any two sets in a generalised orbit of a generalised orbit algebra intersect in the same way every generalised orbit of $\mathcal{D}$, by looking at intersection properties of several generalised orbits.
If $k$ is a positive integer, we consider the tensor product of $k$ copies of $\dign$~: \[\kdign=\underbrace{\dign \otimes \ldots \otimes \dign}_{k}\]
This is a real vector space of dimension $2^{kn}$, with the natural basis~:
\[p_S=p_{S_1}\otimes \ldots\otimes p_ {S_k},\;\textrm{ where }\; S\in\allsets^k\]

\begin{definition} We denote by $\com{k}$ the set of linear functions $h$ from $\kdign$ to $\dign$ such that~:
\[\forall \sigma \in \perms,\ \sigma \circ h = h \circ(\underbrace{\sigma \otimes \sigma \otimes \ldots \otimes \sigma}_{k})\]
\end{definition}

We have already studied the first case ($k=1$) with $\firstcom$. As before, the orbits of $\perms$ on the $(k+1)$-tuples of subsets of $\groundset$ give a basis of $\com{k}$. One can show that the orbit of such a $(k+1)$-tuple 
$\left(S_1,\ldots,S_{k+1}\right)$ is uniquely defined by the function $\mu_{S}$ whose value on every subset $J$ of $\left\{1,\ldots,k+1\right\}$ is 
$\mu_S\left(J\right)=\left\vert \bigcap_{j\in J} S_j\right\vert$. It can be shown that an integer-valued function $\mu$ on $\mathcal{P}\left(\{1,\ldots,k+1\}\right)$ is indeed associated with an orbit if and only if 
\[\forall J\subseteq \{1,\ldots,k+1\}, \sum_{\substack{L\subseteq\{1,\ldots,k+1\}\\L\supseteq J}} (-1)^{\vert L\vert-\vert J\vert} \mu\left(L\right)\geq 0\]

In this case, we shall say that $\mu$ is an \emph{incidence function} of \emph{order} $k+1$.
If a $(k+1)$-tuple $\left(S_1,\ldots,S_{k+1}\right)$ is in the orbit defined by such a function $\mu$, we shall write $\left(S_1,\ldots,S_{k+1}\right)\vdash\mu$. We can also show that the number of orbits is $\binom{n+2^{k+1}-1}{2^{k+1}-1}$.
This enables us to define a basis of $\com{k}$, namely the set of functions 
\[E_\mu:p_{S_1}\otimes \ldots\otimes p_{S_k}\mapsto \sum_{\substack{A\ s.t\ \left(S_1,\ldots,S_k,A\right)\vdash\mu}}p_A \]
where $\mu$ runs over all the incidence functions of order $k+1$.
One can easily see that the $E_{k,l,r}$ we used previously are indeed the $E_\mu$, with $\mu$ running through the incidence functions of order two.

The functions in $\com{k}$ characterise orbit algebras, as~:

\begin{theorem} \label{comkstab} Let $\mathcal{D}$ be a nonempty and nonzero subset of $\dign$ such that for every positive integer $k$, \[\com{k}\left(\mathcal{D}^{\otimes k}\right)\subseteq \mathcal{D}\]
Then $\mathcal{D}$ is an orbit algebra.
\end{theorem}

\begin{proof}
It is easy to see that any such set $\mathcal{D}$ is a generalised orbit algebra and that equality holds. As such, it possesses generalised orbits. Let us consider two sets $A$ and $B$ in the same generalised orbit $\gorb{}$. It is enough to show that there is a permutation of $\perms$ that send $A$ to $B$, and leaves $\mathcal{D}$ invariant.

If we consider a list $S_1,\ldots,S_{2^n}$ of all the subsets of $\groundset$, and a list $\gorb{1},\ldots,\gorb{2^n}$ of their respective generalised orbits (there might be repetitions), we consider $h$ to be the only basis element of $\com{2^n}$ such that $h\left(S_1,\ldots,S_{2^n}\right)=A$. 
We have : $h\left(\mathcal{D}^{\otimes k}\right)\subseteq \mathcal{D}$, so $h\left(\gorb{1},\ldots,\gorb{2^n}\right)$ is a member of $\mathcal{D}$. It has a nonzero coordinate on $\gorb{}$, and we deduce that there exist sets $\left(T_1,\ldots,T_{2^n}\right)$ such that $T_i\in \gorb{i}$ for every $i=1\ldots n$, and $h\left(T_1,\ldots,T_{2^n}\right)=B$.
We claim that there exists a permutation in $\perms$ sending $S_i$ to $T_i$ for every $i$. Such a permutation maps $A$ to $B$ and leaves $\mathcal{D}$ invariant.

\end{proof}

We shall study the first cases of this closeness property, and as we have already proved that it holds if $k=1$ for any generalized orbit algebra, we shall now turn to the case $k=2$.

First, the multiplication $m$ of $\dign$ is a member of $\com{2}$, because every permutation defines an algebra isomorphism, that is, if $p$ and $q$ are two elements of $\dign$, and $\sigma$ is a permutation in $\perms$, one has~: $\sigma\left(p\cdot q\right)=\sigma\left(p\right)\cdot \sigma\left(q\right)$. Also, one can see that if $u,v,w$ are three elements of $\firstcom$, then $\sigma\circ u\left(v(p)\cdot w(q)\right)=u\left(v\circ\sigma\left(p\right)\cdot  w\circ\sigma\left(q\right)\right)$, that is, $u\circ m \circ \left(v\otimes w\right)$ is a member of $\com{2}$.
It is also reassuring to see that for any generalised orbit algebra $\mathcal{D}$, those functions map any element of $\mathcal{D}\otimes\mathcal{D}$ into $\mathcal{D}$. We shall show that the second case of Theorem \ref{comkstab} always hold~:

\begin{theorem}\label{com2} $\com{2}$ is generated, as a real vector space, by the family~:
\[ u \circ m \circ (v\otimes w),\quad \mathrm{ where }\ u,v,w\ \in \firstcom\]
\end{theorem}

\begin{proof}
Let $h=\varepsilon^{-1}\circ m \circ \left( \varepsilon\otimes\varepsilon\right)$ be the only element of $\com{2}$ that maps any element $p_{A}\otimes p_{B}$ of $\dign^{\otimes 2}$ to $p_A$ if $A=B$ and to $0$ otherwise.

\noindent Given any three functions $E_{a_1,k,r_1}, E_{a_2,k,r_2}, E_{k,a_3,r3}$, the function
\[F_{a_1,a_2,a3,r_1,r_2,r_3,k}=E_{k,a_3,r3}\circ h \circ \left( E_{a_1,k,r_1}\otimes E_{a_2,k,r_2}\right)\]
is a linear combination of the functions $E_\mu$, where $\mu$ runs over the incidence function of order three defined by $\mu\left(\{i\}\right)=a_i$ for $i=1,2,3$. We shall show that for a given triple $\left(a_1,a_2,a_3\right)$ of numbers between $0$ and $n$, this linear system gives the $E_{\mu}$.

Given a member $\left(A_1,A_2,A_3\right)$ of the orbit represented by $\mu$, the coefficient of the decomposition of $F_{a_1,a_2,a3,r_1,r_2,r_3,k} $ on $E_\mu$ is the number of subsets $U$ of $\groundset$ such that~:
\begin{eqnarray*}
\vert U\vert&=&k\\
\vert U\cap A_1\vert&=&r_1\\
\vert U\cap A_2\vert&=&r_2\\
\vert U\cap A_3\vert&=&r_3\\
\end{eqnarray*}

If such a set $U$ exists, one must have~:
\[\vert U\cap\left(A_1\cup A_2\right)\vert= \vert U\cap A_1\vert + \vert U\cap A_2\vert - \vert U\cap A_1 \cap A_2\vert \leq \vert U\vert\]
that is to say~:
\[r_1+r_2-k\leq \vert U\cap A_1 \cap A_2\vert \leq \vert A_1 \cap A_2\vert=\mu\left(\{1,2\}\right)\]

Likewise, we have~: 
\[r_1+r_3-k\leq \mu\left(\{1,3\}\right)\]
\[r_2+r_3-k\leq \mu\left(\{2,3\}\right)\]

In the same manner, considering $ U\cap\left(A_1\cup A_2\cup A_3\right) $, one can see that~:
\[\vert U\cap A_1\vert + \vert U\cap A_2\vert + \vert U\cap A_3\vert- \vert U\cap A_1 \cap A_2\vert - \vert U\cap A_1 \cap A_3\vert - \vert U\cap A_2 \cap A_3\vert 
+ \vert U\cap A_1 \cap A_2 \cap A_3 \vert\leq \vert U\vert\]

That is~:
\begin{eqnarray*}
r_1+r_2+r_3-k&\leq &\vert U\cap A_1\cap A_2\cap A_3^c\vert+\vert U\cap A_1\cap A_2^c\cap A_3\vert+\vert U\cap A_1^c\cap A_2\cap A_3\vert\\
&&+2\vert U\cap A_1\cap A_2\cap A_3\vert\\
& \leq & \vert A_1\cap A_2\cap A_3^c\vert+\vert A_1\cap A_2^c\cap A_3\vert+\vert A_1^c\cap A_2\cap A_3\vert+2\vert A_1\cap A_2\cap A_3\vert\\
& \leq & \mu\left(\{1,2\}\right)+\mu\left(\{1,3\}\right)+ \mu\left(\{2,3\}\right) - \mu\left(\{1,2,3\}\right)
\end{eqnarray*}

We summarize these inequalities in matrix form~:
\[ \left[\begin{array}{cccc} 1&1&0&-1\\0&1&1&-1\\1&0&1&-1\\1&1&1&-1\\ 
\end{array} \right]
\cdot
\left[\begin{array}{c}r_1\\r_2 \\r_3\\k\\ 
\end{array} \right]
\leq
\left[\begin{array}{cccc} 1&0&0&0\\0&1&0&0\\0&0&1&0\\1&1&1&-1\\ 
\end{array} \right]
\cdot
\left[\begin{array}{c}\mu\left(\{1,2\}\right)\\ \mu\left(\{1,3\}\right)\\ \mu\left(\{2,3\}\right) \\ \mu\left(\{1,2,3\}\right)
\end{array} \right]\]
It can be checked that the two matrices are invertible.
Thus, given $a_1,a_2,a_3$, if we consider two convenient orders on $r_1,r_2,r_3,k$ and on the incidence functions $\mu$ such that $\mu\left(\{i\}\right)=a_i$, the expression of the $F_{a_1,a_2,a_3,r_1,r_2,r_3,k}$ on the $E_\mu$ is triangular\footnote{to be exact, a subsystem is triangular}, and 
one can see that this system is invertible, that is to say the diagonal elements are non-zero,
by considering $U=\left(A_1\cap A_2\right)\cup\left(A_2\cap A_3\right)\cup\left(A_1\cap A_3\right)$ with the previous notations.

\end{proof}

We would like to use the same strategy to obtain $\com{3}$, using linear combinations of functions such as $u\circ \left(v\otimes id \right)$ where $u$ and $v$ range over $\com{2}$, but this is not possible. In fact, the dimension of the linear hull of such functions is bounded above by three times the squared dimension of $\com{2}$, that is $3\binom{n+7}{7}^2$, whereas the dimension of $\com{3}$ is $\binom{n+15}{15}$. Hence, if $n$ is large enough :
\[3\textrm{dim}\left(\com{2}\right)^2<\textrm{dim}\ \com{3}\]

Theorem \ref{com2} has an interesting consequence : for any generalised orbit algebra $\mathcal{D}$, we have $\com{2}\left(\mathcal{D}\times\mathcal{D}\right)=\mathcal{D}$. We can define $\mathcal{D}$ to be a subalgebra of functions of $\dign$ to itself, via multiplication. Theorem \ref{com2} asserts that this algebra is closed under transposition, and thus completely determined by its commutant. 

\section{Reconstruction problems}

Generalised orbit algebras are strongly related to reconstruction conjectures in graph theory. We state the vertex and edge reconstruction conjectures separately, as they illustrate two distinct points of view on the possible use of generalised orbit algebras in reconstruction problems. We first recall the necessary definitions, the interested reader being invited to refer to \cite{bondy1} for a survey of reconstruction results, and to \cite{bondy2} for an introduction to graph theory.

Let $V=\{1,\ldots,\nvert\}$, and let $E$ be the set of all subsets of $V$ of size two. We shall call $V$ the set of \emph{vertices}, and \emph{edges} the elements of $E$. We define a \emph{graph} to be a subset of $E$, and call $E$ the \emph{complete graph}, whereas the empty set will be the \emph{empty graph}. We also define a \emph{subgraph} of a graph $G$ to be any subset of $G$. The elements of $G$ will be called the edges of $G$, an edge $\{i,j\}$ of $G$ being said \emph{incident} to the vertices $i$ and $j$.

For example, if $\nvert=3$, we can represent the complete graph on three vertices, one of its subgraphs, and an edge incident to the vertices $1$ and $2$. 
\\ \begin{center}\includegraphics[scale=0.75]{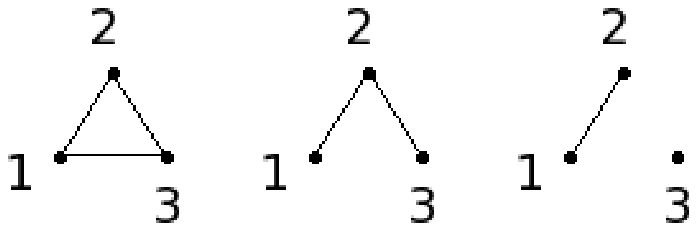}\end{center} 
The group $\permsv$ of permutations of the set $V$ acts on $E$ with the natural action $\sigma\cdot\{i,j\}=\{\sigma\cdot i,\sigma\cdot j\}$. Likewise, we can lift this action of $\permsv$ on edges to the set of all graphs. We will say that two graphs $G$ and $H$ are \emph{isomorphic} if they are in the same orbit for this action. We shall denote by $[G]$ the orbit of the graph $G$.
For example, if $\nvert=3$, the permutation $(1 2)$ acts as follows~:
\\ \begin{center}\includegraphics[scale=0.75]{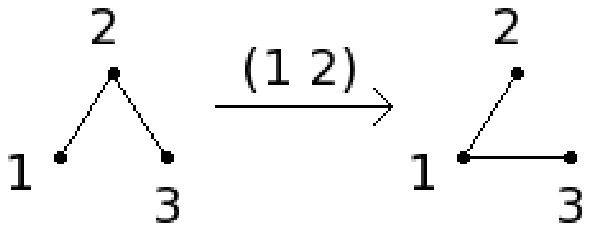}\end{center} 
Together with the following graph, these two graphs form an orbit of $\permsv$~:
\\ \begin{center}\includegraphics[scale=0.75]{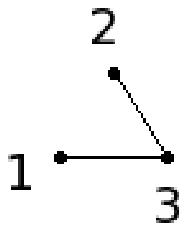}\end{center}

\subsection{Vertex-reconstruction conjecture}

For a graph $G$, and a vertex $v$, we define the \emph{vertex-deleted subgraph} $G-v$ as the subgraph of $G$ that contains every edge of $G$ not incident to $v$.
Two graphs $G$ and $H$ on the same vertex set $V$ are called \emph{hypomorphic} if, for each $v\in V$, the subgraphs $H-v$ and $G-v$ are isomorphic. A \emph{reconstruction} of a graph $G$ is a graph that is hypomorphic to $G$. A graph $G$ is said to be \emph{reconstructible} if every reconstruction of $G$ is isomorphic to $G$. 

\begin{conjecture}[Kelly-Ulam \cite{ulam}]\label{ulam} Every graph on at least three vertices\footnote{with our notations $f\geq 3$} is reconstructible.
\end{conjecture}

For every edge $\{i,j\}$ of $E$, we define a variable $x_{i,j}$, and define the algebra $\mathcal{G}_\nvert$ to be the quotient of the real polynomials $\field[x_{\{1,2\}},\ldots,x_{\{\nvert-1,\nvert\}}]$ by the ideal generated by the elements~: \[x_{\{1,2\}}^2-x_{\{1,2\}},x_{\{1,3\}}^2-x_{\{1,3\}},\ldots,x_{\{\nvert-1,\nvert\}}^2-x_{\{\nvert-1,\nvert\}}\] We reproduce here the algebraic point of view of the first part, in the particular case of graphs, that is to say, when $\groundset=E$.
The real vector space $\galg$ have dimension $2^{\binom{\nvert}{2}}$, and the polynomials 
\[p_G=\prod_{\{u,v\}\in G} x_{\{u,v\}}\]
where $G$ is a graph, are a basis of $\galg$. We have seen that the group of permutations of $V$
$\permsv$ acts on $\galg$ as a group of algebra isomorphisms (subgroup of the group of permutations $\permse$). Consequently, we can consider the algebra of its invariants $\galg^{\permsv}$. It is clear that this is a vector space with basis
\[p_{[G]}=\sum_{G\in [G]} p_G\]
where $[G]$ runs over every isomorphism class of graphs.
If we define by linearity the evaluation of a polynomial $p_G$ on a graph $H$ to be : 
\[p_G(H)=\left\{\begin{array}{ll}1& \textrm{if}\ G\subseteq H\\0&\textrm{otherwise}\end{array}\right.\] 
then $p_{[G]}(H)$ is simply the number of subgraphs of $H$ that are isomorphic to $G$.

An \emph{isolated vertex} of a graph $G$ is a vertex which is incident to no edge of $G$, we denote by $\textrm{iv}\left( G\right)$ the number of isolated vertices in $G$. 
If we now consider the set $\ideal$ of polynomials $p_{[G]}$ in $\galg$, where $[G]$ runs over all the isomorphism classes of graphs with at least one isolated vertex, then $H$ is a reconstruction of $G$ if and only if $p(G)=p(H)$ for every polynomial $p$ in $\ideal$, for~:
\begin{lemma}[Kelly] If $F$ is a graph with at least one isolated vertex, then for every graph $G$~:
\[p_{[F]}\left(G\right)=\frac{1}{\textrm{iv}(F)}\sum_{v\in V}p_{[F]}\left(G-v\right)\]
\end{lemma}

Knowing the values of the polynomials $p$ of $\ideal$ on $G$ allows us to know the list of the vertex deleted subgraphs $G-v$ up to isomorphism.
Thus, the generalised orbits of the partition of the set of graphs associated with the algebra generated by $\ideal$ in Theorem \ref{partition} are the classes of hypomorphic graphs.
We deduce that Conjecture \ref{ulam} is equivalent to~:
\begin{conjecture}\label{ulamalg}If $\nvert$ is at least three, then the subalgebra of $\galg$ generated by the polynomials of $\ideal$ is $\galg$ itself.
\end{conjecture}

Thus we can split the Kelly-Ulam Conjecture \footnote{conjecture \ref{ulam}} into three parts~:

for $\nvert\geq 3$,
\begin{itemize}
\item The subalgebra of $\galg$ generated by the polynomials of $\ideal$ is a generalised orbit algebra.
\item This generalised orbit algebra is an orbit algebra.
\item The stabilizer of this algebra is $\permsv$.
\end{itemize}

We can easily address the third one, that is to say the group of permutations of the edges that leave invariant every polynomial of $\ideal$ is $\permsv$.
To show this, we note that a permutation of $\permse$ that leaves any polynomial of $\ideal$ invariant also leaves invariant the polynomial :
\[g=\sum_{i=1}^\nvert \prod_{\substack{j=1\\j\neq i}}^n x_{\{i,j\}}\]
because $\complement(g)$ is in $\ideal$.

A consequence of the definition of the action of $\permse$ on $\galg$ is that such permutations permute the 
\[s_i= \prod_{\substack{j=1\\j\neq i}}^n x_{\{i,j\}}\]
This yields a one to one correspondence between the permutations of $E$ that leave $g$ (i.e. $\ideal$) invariant, and the permutations of $V$. 

Observe that if the second point is true the generalised orbit algebra should be generated by $g$ \footnote{As the intersection of two generalised orbit algebras is a generalised orbit algebra, we may consider the generalised orbit algebra generated by $g$ to be the smallest generalised orbit algebra containing $g$.}.
It is clear that the vertex reconstruction conjecture implies the second point. But one can also show :

\begin{theorem} If the vertex reconstruction conjecture is true, then the generalised orbit algebra generated by $g$ is the orbit algebra of graphs.
\end{theorem}

\begin{proof}
We prove that for every type of graph $[G]$, $p_{[G]}$ is in the generalised orbit algebra generated by $g$, by induction on the number of non-isolated vertices of $G$.
We have seen that $p_{\emptyset}=1$ is in every generalised orbit algebra, just like the polynomial associated with graphs of size one, that is to say edges. Now if we consider an isomorphism class of graph $[G]$ with $t+1$ non-isolated vertices ($t\geq 2$), we know by the induction hypothesis that the polynomials $p_{[H]}$, where $H$ is a graph with at most $t$ non-isolated vertices, are elements of the generalised orbit algebra generated by $g$. By the reformulation of Conjecture \ref{ulam} as Conjecture \ref{ulamalg} (with $f=t+1$), there exists a polynomial $r_{[G]}$ in the generalised orbit algebra generated by $g$, that takes the same value as $p_{[G]}$ on every graph with at most $t+1$ non-isolated vertices.
Observe that $r_{[G]}$ is equal to $p_{[G]}$ plus a sum of $\alpha_{[H]}p_{[H]}$, where $[H]$ runs through the isomorphism classes of graphs with at least $t+2$ non-isolated vertices. We conclude by eliminating these terms as follows~:
\begin{itemize}
\item First we generate the polynomial $p_{t+1}$ corresponding to the cliques of size $t+1$ as a member of the generalised orbit algebra generated by $g$.
\item Then we note that for any graph $H$ with less than $t+1$ non isolated vertices~:
\[\varepsilon_H \cdot \left(\varepsilon\circ l(p_{t+1})\right)=\binom{\textrm{iv}(H)}{\nvert-t-1}\varepsilon_H\]
and that the left hand side is otherwise zero.
\item Finally, we consider $\varepsilon^{-1}\left(\varepsilon(r_G)\cdot\left(\varepsilon\circ l(p_{t+1})\right)\right)$. 
\end{itemize}
\end{proof}  

As infinite families of non-reconstructible digraphs exist (see \cite{stockmeyer}), we briefly look at what happens in this case. Instead of considering $\groundset$ to be the set of all two element subsets of $\{1,\ldots,\nvert\}$, we consider it to be the set $P$ of ordered pairs of distinct elements. The group of permutations $\permsv$ acts on $P$ by $\sigma\cdot(i,j)=(\sigma\cdot i,\sigma\cdot j)$.
One can then ask whether the polynomials corresponding to isomorphism classes of digraphs with at least one isolated vertex generate the whole orbit algebra.
\begin{itemize}
\item if $\nvert=3$, the algebra generated by the polynomials corresponding to types of digraphs with at least one isolated vertex is a generalised orbit algebra, and an orbit algebra, but the stabilizer also switches edges directions.
\item if $\nvert=4$, the algebra is not a generalised orbit algebra. Consider the following two hypomorphic digraphs: there are no other digraph hypomorphic to them~:
\begin{center}\includegraphics[scale=0.75]{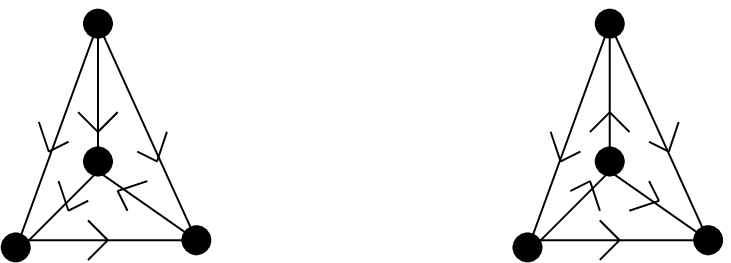}\end{center}
Both of these digraphs contain the second of the following two digraphs, but not the first.
\begin{center}\includegraphics[scale=0.75]{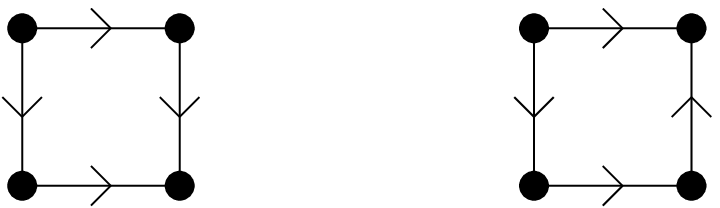}\end{center}  
Those last two digraphs are nevertheless hypomorphic, and there are no other digraph hypomorphic to them. We see here that the number of subdigraphs with a given deck is not reconstructible, which implies that the algebra generated by the polynomials corresponding to types of digraphs with at least one isolated vertex is not a generalised orbit algebra
\end{itemize}

\subsection{Edge-reconstruction conjecture}

If $G$ is a graph, a graph $H$ is said to be an \emph{edge-reconstruction} of $G$ if there is a bijection $\phi$ between the edges of $G$ and the edges of $H$ such that for every edge $u$ of $G$, $G-u$ and $H-\phi(u)$ are isomorphic. The graph $G$ is said to be \emph{edge-reconstructible} if every edge-reconstruction of $G$ is isomorphic to $G$. 

\begin{conjecture}[Harary\cite{harary}]\label{edgerec} Every graph on at least four edges is edge-reconstructible.
\end{conjecture}

We take the opportunity here to emphasize the structure coefficients of a generalised orbit algebra. Recall that there is a partition of $\allsets$\footnote{called system of generalised orbits} associated with any generalised orbit algebra by Theorem \ref{partition}, that is both a basis of the generalised orbit algebra, and the list of sets on which the polynomials in the generalised orbit algebra take constant values. We deduce that for any generalised orbit algebra $\mathcal{D}$ whose system of generalised orbits is 
$\gorb{1},\ldots,\gorb{s}$, one can define, for any two generalised orbits $\gorb{i}$ and $\gorb{j}$, and any member $A$ of $\gorb{i}$~: 
\[\binom{\gorb{i}}{\gorb{j}}=\left\vert\left\{B\subseteq A\ s.t.\ B\in \gorb{j}\right\}\right\vert\]
as this number is independent of the choice of $A$ in $\gorb{i}$ (we might use the different notations $\binom{A}{\gorb{j}}$ or even $\binom{A}{B}$ for this number). These coefficients encode structure of the generalised orbit algebra $\mathcal{D}$ in the following sense~:
\begin{itemize}
\item We consider the list of the generalised orbits of the generalised orbit algebra to be such that the matrix $\complement$ is known, for example with $p_{\gorb{i}^c}:=\complement\left(p_{\gorb{i}}\right)=p_{\gorb{s-i}}$.
\item The matrix of $\partial$ with respect to the basis $p_{\gorb{1}},\ldots,p_{\gorb{s}}$ has coefficient $[\gorb{i},\gorb{j}]$ equal to $\binom{\gorb{i}^c}{\gorb{j}^c}$ if $\sharp \gorb{i}=\sharp \gorb{j} -1$ and $0$ otherwise.  
\item The multiplication rule is given by 
\[p_{\gorb{i}}\cdot p_{\gorb{j}}= \sum_{k=1}^s \binom{\gorb{k}}{\gorb{i},\gorb{j}}p_{\gorb{k}}\]
where the coefficient $\binom{\gorb{k}}{\gorb{i},\gorb{j}}$ can be computed via M\"obius inversion~:
\[\binom{\gorb{k}}{\gorb{i},\gorb{j}}=\sum_{l=1}^s (-1)^{\sharp \gorb{k} - \sharp \gorb{l}}\binom{\gorb{k}}{\gorb{l}}\binom{\gorb{l}}{\gorb{i}}\binom{\gorb{l}}{\gorb{j}}\]
\end{itemize}

For any generalised orbit algebra $\mathcal{D}$, we define a \emph{coefficient matrix} to be a matrix indexed by a list of generalised orbits $\left(\gorb{1},\ldots,\gorb{s}\right)$ of $\mathcal{D}$, and with coefficient $[i,j]$ equal to $\binom{\gorb{i}}{\gorb{j}}$.

A Gap (\cite{gap}) experiment shows that an orbit algebra of order smaller than 9 is uniquely defined up to conjugation by its coefficient matrix. It is then natural to look at which matrices are indeed coefficient matrices of a generalised orbit algebra. We can easily state a number of relations between coefficients. For example, a simple counting argument gives~:
\[\binom{\groundset}{\gorb{i}}\binom{\gorb{i}^c}{\gorb{j}^c}=\binom{\groundset}{\gorb{j}}\binom{\gorb{j}}{\gorb{i}}\] 

We can also see that the coefficient matrix $\mathcal{M}$ is the matrix of $\complement\circ l\circ \complement$. We can now reformulate Lemma \ref{lpuiss} in its original form (\cite{mnukhin2}), that is to say, for every nonzero integer $m$~:
\begin{eqnarray}\label{mnukhin}
\mathcal{M}^m_{[i,j]}=m^{\sharp \gorb{i} - \sharp \gorb{j}}\mathcal{M}_{[i,j]}
\end{eqnarray}
With the study of the $\com{k}$, we have already defined many relations~:
\begin{itemize}
\item relations that result from the definition of $\firstcom$ as the algebra generated by $\partial$ and $\complement$.
\item composition relations between the elements of the basis of $\firstcom$.
\item linear combinations between the $u\circ m\circ(v\otimes w)$, where $u,v,w$ runs through the $E_{k,l,r}$ (i.e. a basis of $\firstcom$).
\item although we do not know how to generate $\com{k}$ if $k\geq 3$, some linear relations might arise between the functions that we do know how to generate.
\end{itemize}

As we can see, the framework of the commutants ($\firstcom$, $\com{k}$) generates ex-nihilo some polynomial relations between the coefficients of a generalised orbit algebra. We might ask whether they can be useful towards reconstruction problems.
First, we follow \cite{cameron} and rephrase the Edge Reconstruction Conjecture \ref{edgerec} in the more general framework of orbit algebras.

Given a group $\Gamma$ of permutations of $\groundset$, we say that two subsets $A$ and $B$ of $\groundset$ are \emph{$\Gamma$-isomorphic} if there is an element $\sigma$ of $\Gamma$ such that $\sigma\cdot A=B$. 
If $A$ and $B$ are two subsets of $\groundset$ and there is a bijection $\phi$ from $A$ to $B$ such that for every $e$ in $A$, the subsets $A-e$ and $B-\phi(e)$ are $\Gamma$-isomorphic, we say that $B$ is a \emph{$\Gamma$-reconstruction} of $A$. $A$ is said to be \emph{$\Gamma$-reconstructible} if every $\Gamma$-reconstruction of $A$ is isomorphic to $A$. We would like to know which sets are $\Gamma$-reconstructible ? 
There is a version of Kelly's Lemma for this problem~:
\begin{lemma}[Kelly]
If $A$ and $C$ are two subsets of $\groundset$ such that $\vert C\vert<\vert A\vert$, then~:
\[\binom{A}{C}=\frac{1}{\vert A\vert -\vert C\vert}\sum_{e\in A} \binom{A-e}{C}\] 
\end{lemma}

It is then clear that requiring that two sets $A$ and $B$ satisfy the condition that $A-e$ is isomorphic to $B-\phi(e)$ for every element $e$ of $A$ is equivalent to requiring that, 
for every set $C$ with fewer than $\vert A\vert$ elements,
\[\binom{A}{C}=\binom{B}{C}\]
 
Note also that Kelly's Lemma apply for generalised orbit algebras. One can show that the relations stated in this result are exactly the same as the ones obtained in equation (\ref{mnukhin}), that is to say, a coefficient matrix satisfying one set of relations satisfies the other set.
 
We now look at the first relation obtained with the generalised orbit algebra structure. From equations (\ref{bascom1}) or (\ref{bascom2}), we deduce that for any generalised orbit algebra $\mathcal{D}$, and for any two generalised orbits $\gorb{1}$ and $\gorb{2}$ of $\mathcal{D}$ with sizes $k$ and $l$ respectively, the coefficient of the matrix of $E_{k,l,r}$ is~:
\[E_{k,l,r_{[\gorb{1},\gorb{2}]}}=\sum_{\bl{U}} (-1)^{\vert \bl{U}\vert -r}\binom{\vert \bl{U}\vert}{r}\binom{\gorb{1}}{\bl{U}}\binom{\bl{U}^c}{\gorb{2}^c}\]
where the sum runs over every generalised orbit $\bl{U}$ of $\mathcal{D}$.
Observe that this should be $0$ if $k+l-r>n$. Thus if $\gorb{1}$ and $\gorb{2}$ are two generalised orbits both of size $k>\frac{n}{2}$ such that for every generalised orbit $\bl{U}$ distinct from $\gorb{1}$ or $\gorb{2}$ we have $\binom{\gorb{1}}{\bl{U}}=\binom{\gorb{2}}{\bl{U}}$, then
\begin{eqnarray*}
E_{k,k,0_{[\gorb{1},\gorb{1}]}}-E_{k,k,0_{[\gorb{2},\gorb{1}]}}&=&\sum_{\bl{U}} (-1)^{\vert \bl{U}\vert}\left(\binom{\gorb{1}}{\bl{U}}-\binom{\gorb{2}}{\bl{U}}\right)\binom{\bl{U}^c}{\gorb{1}^c}\\
&=&(-1)^{\vert \gorb{1}\vert}
\end{eqnarray*}
Since $E_{k,k,0}$ is the zero mapping we have a contradiction.
We deduce that~:
\begin{theorem}[Lov\'asz \cite{lovasz}\cite{cameron}]\label{lovasz}
For every group $\Gamma$ of permutations of $\groundset$, if $A$ and $B$ are two subsets of $\groundset$ and there is a bijection $\phi$ from $A$ to $B$ such that for every $e$ in $A$, $A-e$ and $B-\phi(e)$ are $\Gamma$-isomorphic, and if the size of $A$ is bigger than $\frac{n}{2}$, then $A$ and $B$
are $\Gamma$-isomorphic.
\end{theorem}

This result also applies to generalised orbit algebras, if one replaces the notion of isomorphism by membership in the same generalised orbit.
This theorem is in some sense best possible in the general framework of orbit algebras, and consequently, also for generalised orbit algebras. 
Indeed, there exists orbit algebras of order $2r$ with non reconstructible sets of size $r$. As an example, consider the orbit algebra of the permutation group $\Gamma$ generated by~:
\[ \left\{(1,2)(2i+1,2i+2),\quad i=1..r-1   \right\}\]

Consider the set $U=\left\{2i,\ i=2\ldots r\right\}$, and let $A=U\cup \{1\}$, $B=U\cup\{2\}$. We now have two sets $A$ and $B$ of size $r$ not in the same orbit of $\Gamma$, because every element of $\Gamma$ leaves invariant the parity of the number of even elements of sets containing exactly one element in each pair $\{2i+1,2i+2\},\ i=0\ldots r-1$. However, if one considers a subset\footnotemark[1] of size $r-1$ in $A$, then one can see that there exists exactly one pair $\{2j+1,2j+2\}$ with no element in $A$, and applying $(1,2)(2j+1,2j+2)$ we find a subset\footnotemark[1] of size $r-1$ in $B$. This defines a one to one mapping from the subsets of size $r-1$ in $A$ to the subsets of size $r-1$ in $B$ 
\footnotetext[1]{distinct from $U$}, showing that $B$ is a $\Gamma$-reconstruction of $A$.

\noindent For the sake of completeness, one can add the element $2r+1$, and get an algebra of order $2r+1$ with non-reconstructible sets of size $r$. 

It is therefore natural to try to find properties of the group that would allow us to lower the bound of $\frac{n}{2}$. Considering the order of the group yields a theorem of V.M\"uller that also applies in any generalised orbit algebra (even if there is no group to consider). We prove a slightly different result in the more general context of generalised orbit algebras~:

\begin{theorem}\label{mull}
If $\mathcal{D}$ is a generalised orbit algebra of order $n$ with generalised orbits $\gorb{1},\ldots,\gorb{s}$,  and $A$ and $B$ are two subsets of $\groundset$ with cardinality $k$, such that for every generalised orbit $\gorb{i}$ of size less than $k$, $\binom{A}{\gorb{i}}=\binom{B}{\gorb{i}}$, then for every generalised orbit $\gorb{j}$ :
\[2^{k-\vert \gorb{j}\vert-1}\leq  \binom{\gorb{j}^c}{A^c}\]
\end{theorem}

\begin{proof}

Let $A$ be a set and $B,S$ generalised orbits, then
\[E_{A,B}^S=\sum_{V\in\gorb{}}(-1)^{\vert V\vert -\vert S\vert}\binom{V}{S}\binom{A}{V}\binom{V^c}{B^c}\]
is the number of graphs whose intersection with $A$ is exactly a copy of $S$, and that are elements of $B$.

We have~:
\[\sum_S \binom{S}{T}E_{A,B}^S=\binom{A}{T}\binom{T^c}{B^c}\]
and if B is a reconstruction of A:
\begin{eqnarray*}
E_{A,A}^T-E_{B,A}^T&=&\sum_{V\in\gorb{}}(-1)^{\vert V\vert -\vert T\vert}\binom{ V}{T}\left(\binom{A}{V}-\binom{B}{V}\right)\binom{V^c}{A^c}\\
&=&(-1)^{\vert A\vert -\vert T\vert}\binom{ A}{T}\\
\end{eqnarray*}

so 
\begin{eqnarray*}
2^{\vert A\vert-\vert S\vert}\binom{A}{S}&=&\sum_T \binom{T}{S}\vert E_{A,A}^T-E_{B,A}^T \vert\\
& \leq &\sum_T  \binom{T}{S}E_{A,A}^T +\sum_T  \binom{T}{S}E_{B,A}^T\\
&\leq&\binom{A}{S}\binom{S^c}{A^c}+\binom{B}{S}\binom{S^c}{A^c}\\
\end{eqnarray*}
hence, if $S$ contains a strict subset of $A$ (thus, of $B$), we have~:
\[2^{\vert A\vert-\vert S\vert-1}\leq  \binom{S^c}{A^c}\]

\end{proof}
Using $\gorb{j}=\left\{\emptyset\right\}$ we get~:
\begin{corollary}[M\"uller \cite{muller}]\label{muller}
If $\mathcal{D}$ is a generalised orbit algebra of order $n$ with generalised orbits $\gorb{1},\ldots,\gorb{s}$,  and $A$ and $B$ are two subsets of $\groundset$ with cardinality $k>1+\log_2\binom{\groundset}{A}$, such that for every generalised orbit $\gorb{i}$ of size less than $k$, $\binom{A}{\gorb{i}}=\binom{B}{\gorb{i}}$, then $A$ and $B$ belong to the same generalised orbit of $\mathcal{D}$.
\end{corollary}

In the case of the orbit algebra of the group $\Gamma$, this implies that sets of size greater than $1+\log_2 \vert \Gamma\vert$ are $\Gamma$-reconstructible (see \cite{cameron}).
In the above example, the group $\Gamma$ is commutative, and of order $2^{r-1}$, so $1+\log_2\left(\Gamma\right)=r$ : this example also shows that M\"uller's theorem is best possible. By adding an appropriate number of (fixed) elements, one can construct a subgroup of $\perms$ of order $2^{r-1}$ with non-reconstructible subsets of size $r$, if $n\geq 2r$. Thus, all the limit cases of M\"uller's theorem are covered, because if $n<2r$, we know by Lov\'asz's theorem that sets of size $r$ are reconstructible.
This example shows that one cannot use the polynomial relations between the coefficients of generalised orbit algebras to improve Theorems \ref{lovasz} and \ref{muller} without introducing another generalised orbit algebra parameter than the maximum size of a generalised orbit.  

As an example, one can consider the following :

\begin{corollary}[Maynard-Siemons \cite{maynard0}]
If $\Gamma$ acts freely on $\Omega$, then the reconstruction index of $\Gamma$, defined as the least cardinality for which every set is $\Gamma$-reconstructible, is at most $5$.
\end{corollary}

\begin{proof}
Using $\gorb{j}$ as a convenient orbit of cardinality $1$, we have $2^{\vert A\vert -2}\leq \frac{\binom{\Omega}{A}}{\binom{\Omega}{\gorb{j}}}\binom{A}{\gorb{j}}$.
As $\Gamma$ acts freely we have $\vert \Gamma\vert=\binom{\Omega}{\gorb{j}}\geq \binom{\Omega}{A}$.
We deduce that $2^{\vert A\vert -2}\leq \vert A\vert$, so $\vert A\vert \leq 4$.
\end{proof}

We refer to \cite{maynard0} for a complete classification of freely acting groups with respect to their reconstruction index.




\subsection{Generalised orbit algebras are not orbit algebras}

To construct generalised orbit algebras not arising from a group of permutations, we go back to Theorem \cite{lovasz}, and remark that according to the proof, 
any two sets $A$ and $B$ such that for every orbit distinct from $A$ and $B$ $\binom{A}{C}=\binom{B}{C}$ have size equal to $r$ only if either :
\begin{itemize}
\item elements of $A$ and $B$ are complements of each other, or
\item A and B are self-complementary
\end{itemize}

There exists invariant algebras for both cases, in the above example, the orbits of $A$ and $B$ are~: 
\begin{itemize}
\item complements of each other if $r$ is odd
\item both self complementary if $r$ is even
\end{itemize}

In these conditions, we remark that unifying the generalised orbits $A$ and $B$ yields a generalised orbit algebra. Sometimes, the resulting generalised 
orbit algebra may not be an orbit algebra, that is to say, for the first orbit algebra, there's not always an outer permutation
stabilizing every orbit except A and B, but mixing elements of A and B.

For example, if $n=8$ and $\Gamma$ is the group generated by the permutations \[(1,2)(3,4),\ (5,6)(7,8),\ (1,3,2,4)(5,7,6,8),\ (1,5)(2,6)(3,7)(4,8)\]
we take $A=\{1,3,5,7\}$ and $B=\{1,3,5,8\}$.
$A$ can be written as the union of two sets in the same orbit $O$ : $\{1,3,7\}$ and $\{3,5,7\}$, with intersection $\{3,7\}$, whereas $B$ can be written only in one way as the union of two sets of $O$ : $\{1,3,8\}$ and $\{1,5,8\}$, but the intersection $\{1,8\}$ is not in the same orbit as $\{3,7\}$.

This is in contradiction with the fact that there can be a permutation
mapping $\{1,3,5,7\}$ to $\{1,3,5,8\}$ respecting the orbits of
$\{1,3,7\}$,,$\{3,7\}$, and $\{1,8\}$.
Thus, the resulting generalised orbit algebra is not an orbit algebra, as it is not "closed" under $\com{3}$, in the sense of Theorem \ref{comkstab}.

\subsection*{Acknowledgements}
First, I am very indebted to an anonymous referee for his helpful remarks and suggestions. I also thank Guus Regts for some very interesting discussions.

Some of the results presented here were obtained during my PhD Thesis at the University Claude Bernard, Lyon, France. I would like to thank Pr J.A. Bondy for his supervision, teaching, and for his invaluable help with the redaction of this text. The counterexample was found recently during a post-doc stay at CWI.


\begin{thebibliography}{99}
\bibitem{bondy1}J.A.Bondy, A graph reconstructor's manual. \emph{Surveys in Combinatorics} (1991) $\mathbf{166}$ 221-252
\bibitem{bondy2}J.A.Bondy U.S.R.Murty, Graph Theory. \emph{Graduate Texts in Mathematics} $\mathbf{244}$ Springer (2007)
\bibitem{cameron} P.J.Cameron, Stories from the age of reconstruction.\emph{Congressus Numerantium} $\mathbf{113}$ (1996) 31-41
\bibitem{cameron2} P.J.Cameron P.M.Neumann J.Saxl, On groups with no regular orbits on the set of subsets. \emph{Archiv der Mathematik} $\mathbf{43}$ (1984) no. 4, 295-296.
\bibitem{gap} The GAP~Group, \emph{GAP -- Groups, Algorithms, and Programming, Version 4.4.12}; 2008, \verb+(http://www.gap-system.org)+.
\bibitem{go} J.T.Go, The Terwilliger algebra of the hypercube. \emph{European J. Combin.} $\mathbf{23}$ (2002) no.4. 399-429
\bibitem{harary}F.Harary, On the reconstruction of a graph from a collection of subgraphs. \emph{Theory of graphs and its applications} (1964)
\bibitem{livingstone} D.Livingstone A.Wagner, Transitivity of finite permutation groups on
unordered sets. \emph{Mathematische Zeitschrift} $\mathbf{90}$ (1965) 393-403
\bibitem{lovasz} L.Lov\'asz, A note on the line reconstruction problem. \emph{J. Combin. Theory Ser. B} $\mathbf{13}$ (1972) 309-310
\bibitem{maynard0} P.Maynard J.Siemons, On the reconstruction index of permutation groups : semiregular groups. \emph{Aequationes Math.} $\mathbf{64}$ (2002) 218-231
\bibitem{maynard} P.Maynard J.Siemons, On the reconstruction index of permutation groups : general bounds. \emph{Aequationes Math.} $\mathbf{70}$ (2005) 225–239
\bibitem{mnukhin} V.B.Mnukhin, An introduction to M\"obius algebras. \emph{Tempus Lecture notes} $\mathbf{11}$
\bibitem{mnukhin2} V.B.Mnukhin, The k-orbit reconstruction and the orbit algebra. \emph{Acta Applic. Math.} $\mathbf{29}$ (1992) 83-117
\bibitem{muller} V.M\"uller, The edge reconstruction hypothesis is true for graphs with more then $n \log_2 n$ edges. \emph{J. Combin. Theory Ser. B}  $\mathbf{22}$ (1977) 281-283
\bibitem{schrijver} A.Schrijver, New code upper bounds from the Terwilliger algebra and semidefinite programming. \emph{IEEE Trans. Inform. Theory} $\mathbf{51}$ (2005), no. 8, 2859-2866
\bibitem{stockmeyer} P.K.Stockmeyer, A census of nonreconstructible digraphs.I. Six related families. \emph{J.Combin.theory Ser. B}  (1981) $\mathbf{31}$ 232-239
\bibitem{ulam}S.M.Ulam, A Collection of Mathematical Problems. Wiley (Interscience), New York (1960) $\mathbf{29}$

\end{thebibliography}
\end{document}